\documentclass[12pt]{article}
\usepackage{amssymb}	% Allows use of AMS's mathematical symbols
\usepackage{amsmath}    % Allows use of AMS's mathematical commands
\usepackage{epsfig}

\newenvironment{proof}[1][Proof:]{\begin{trivlist} 
\item[\hskip \labelsep {\bfseries #1}]}{\end{trivlist}} 
\newcommand{\qed}{\nobreak \ifvmode \relax \else \ifdim\lastskip<1.5em \hskip-\lastskip \hskip1.5em plus0em minus0.5em \fi \nobreak \vrule height0.75em width0.5em depth0.25em\fi} 
\def\0{\bf \0}

\def\0{{\bf 0}}

\def\R{\mathbb{R}}

\def\Z{{\bf Z}}

\def\Tr{{\rm T}}

\newtheorem{example}{Example}[section]
\newtheorem{theorem}{Theorem}[section]
\newtheorem{lemma}{Lemma}[section]

\newtheorem{remark}{Remark}[section]

\usepackage{color}
\usepackage[normalem]{ulem}

\begin{document}
\title{A note on the diameter of convex polytope}
\author{
Yaguang Yang\thanks{US NRC, Office of Research, 
11555 Rockville Pike, Rockville, 20850. 
Email: yaguang.yang@verizon.net.} 
}

\date{\today}

\maketitle    % This command generates the title.

\begin{abstract}
This short note extends a recent result (Bonifas et al,
On sub-determinants and the diameter of polyhedra, 
Discrete Computational Geometry, 52, 2014) of an upper 
bound of the diameter of a convex polytope defined 
by an integer matrix to a similar upper bound of the 
diameter of a convex polytope defined by a real matrix. 
It also shows, by an example, that the new bound
may be better than the ones of Bonifas et al.
\end{abstract}

{\bf Keywords:} diameter of convex polytope, linear programming.

%{\bf MSC classification:} 90C05 Linear programming.
%\newpage
 
\section{Introduction}

A polytope $P=\{ x\in \R^n: Ax \le b \}$ is defined
by a $m \times n$ matrix $A$, a vector $b$, and $m>n$. 
Let $x^* \in P$ denote a vertex of $P$ which satisfies 
(a) the system of inequalities $A x^* \le b$ holds and 
(b) $n$ equalities hold for some linear independent rows 
of $A$. Two vertices $x^*$ and $y^*$ are neighbours if 
they are connected by an edge of $P$, which is defined by
$n-1$ linearly independent rows of $A$ where
the equalities hold for both $x^*$ and $y^*$. In 
this way, any two vertices on $P$ are connected
by a path composed of a series of edges.
The diameter of $P$ is the integer that is the smallest
number of edges between any two vertices on $P$, which
defines the shortest path between $x^*$ and $y^*$.

The famous Hirsch conjecture (see \cite{dantzig63})
states that for $m>n \ge 2$, diameter of $P$ is less than 
$m-n$. After 50 years of extensive research, this 
conjecture was disproved by Santos \cite{santos12}.
But the interest on the bound of the diameter of 
polytope is not reduced because this problem is not
only hard but also has theoretical implication to the 
simplex method of the linear programming 
\cite{blf17,dm16,dp18,sukegawa17}. Recently,
Bonifas et al. \cite{bsehn14} derived an upper bound for
a polytope with total unimodularity,\footnote{
Although Bonifas et al. assumed that $A \in \Z^{m \times n}$,
their results are applicable to more general settings as
we will see in the derivations of this note.} i.e., for 
$A \in \Z^{m \times n}$, the upper bound is given as
$\mathcal{O} \left( n^{3.5} \Delta^2 \ln(n\Delta) \right)$, 
where $\Delta$ is the largest absolute value among all 
$(n-1) \times (n-1)$ sub-determinants of $A$. 
This short note shows that their method can be extended 
to the case where $A \in \R^{m \times n}$. Moreover, 
if $A \in \Z^{m \times n}$, it also shows, by an example,
that the new bound may be better than the bound of 
\cite{bsehn14}. We would also like to point out that
parameters other than m and n 
(for example, smoothness parameters)  in iteration bound 
for simplex method have been studied \cite{dh17}.

Without loss of generality, we may assume that the
lengths of all row vectors of $A$ are one, which can easily 
be achieved by normalizing the row $A_i$, the $i$th
row of $A$, and dividing $b_i$ by $\| A_i \|$ for all $i$. 
This does not change the graph of the polytope $P$.

\section{Main results}

We follow the notations and definitions of Bonifas 
et al. \cite{bsehn14}. First, assume that $P$ is
non-degenerate, i.e., each vertex has exactly $n$
tight inequalities. % (in other word, $n$ facets). 
Let $V$ be the set of all vertices of $P$.
The normal cone $C_v$ of a vertex $v$ is the set 
of all vectors $c \in \R^n$ such that $v \in V$ 
is an optimal solution of the linear programming
$\max \{ c^{\Tr} x: x\in \R^n, Ax \le b \}$. Two 
vertices $u$ and $v$ are adjacent if and only if
$C_u$ and $C_v$ share a facet. Let the unit ball
\[
B_n= \{ x \in \R^n: \| x \|_2 \le 1 \}.
\] 
The volume of the union of the normal cones of 
$U \subseteq V$ is defined as 
\[
vol (S_U)= vol \left( \cup_{v \in U} C_v \cap B_n \right),
\]
where $S_v=C_v \cap B_n$ is defined as the sphere cone
of $C_v$. 

For any two vertices $u$ and $v$ in $P$, starting from 
$u$ and $v$, the breadth-first-search finds all the
neighbour vertices by iteration until a common vertex
is discovered. The shortest path is no more than two
times the number of iterations. Let $I_j \subseteq V$ be the 
set of vertices that have been discovered in $j$th 
iteration. Clearly, if 
\begin{equation}
vol(S_{I_j}) \ge \frac{1}{2}vol(B_n),
\label{basic}
\end{equation}
then, the common vertex must be found in less than
$j$ iterations, i.e., the diameter is bounded by $2j$.
The rest effort is to estimate $j$ such that equation
(\ref{basic}) holds.

The $(n-1)$-dimensional surface of a spherical
cone $S$ that is not on the sphere is denoted as
the dockable surface $D(S)$. 
%The base of $S$
%is the intersection of the sphere of $S$ with the
%unit sphere. The area of the base of $S$ is denoted
%as $B(S)$ and the length of the relative boundary of
%$S$ is defined by $L(S)$. The following relations are 
%well-known:
%\[
%vol(S)=\frac{B(S)}{n}, \hspace{0.1in}
%D(S)=\frac{L(S)}{n-1}.
%\]
Bonifas et al. showed the following:
\begin{lemma}[Bonifas et al. \cite{bsehn14}]
Let $S$ be a (not necessarily convex) spherical cone with
$vol(S) \le \frac{1}{2} vol (B_n)$. Then,
\begin{equation}
\frac{D(S)}{vol(S)} \ge \sqrt{\frac{2n}{\pi}}.
\end{equation}
\label{lemma1}
\end{lemma}

\begin{figure}[htb]
\centerline{\includegraphics[height=6cm,width=6.5cm]
{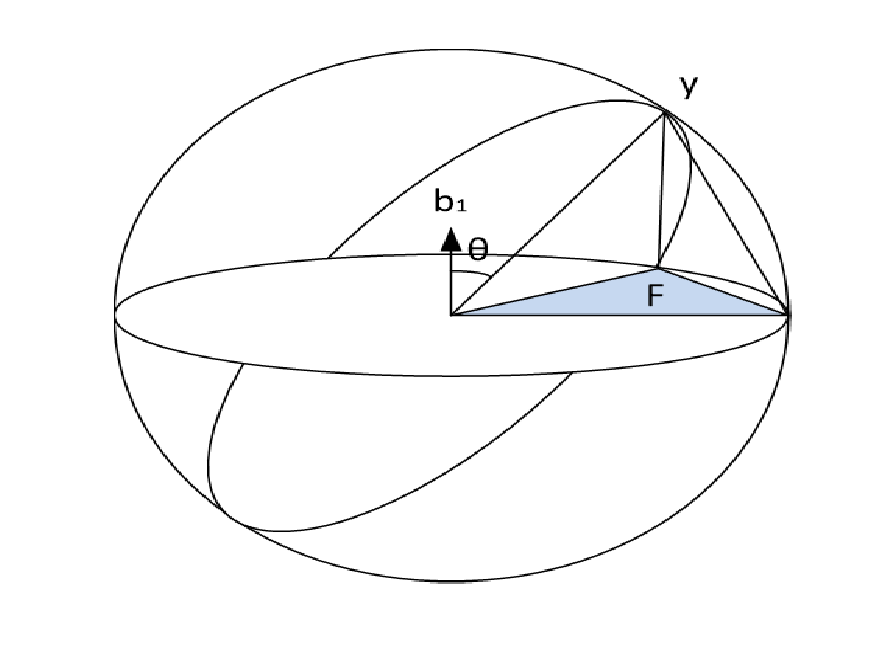}}
\caption{Proof of Lemma 2.2.}
\label{Lemma2}
\end{figure}

Let $\Delta$ denote the largest absolute value among all 
$(n-1) \times (n-1)$ sub-determinants of $A$ and $A_v$ 
be a $n \times n$ matrix of $A$ corresponding to a
vertex $v \in V$, i.e., there is a $x$ satisfying
$Ax \le b$ and $A_v x = b_v$ where $b_v$ is 
a sub-vector of $b$ whose index set is the same 
as $A_v$. Denote $\det(A^*)=\min_{v \in V} \det(A_v)$,
where $\det(A_v)$ is the volume of the box spanned by the 
(unit length) row vectors of $A_v$. $\det(A^*)$
can be viewed as the condition number of polytope
\cite{yang89}. The next lemma is a modification of 
Lemma 3 of Bonifas et al. \cite{bsehn14}.
\begin{lemma}
Let $v$ be a vertex of $P$. Then, one has,
\begin{equation}
\frac{D(S_v)}{vol(S_v)} \le 
\frac{n^{2.5} \Delta }{\det(A^*)}.
\end{equation}
\label{lemma2}
\end{lemma}
\begin{proof}
The proof uses the same idea of Bonifas et al. 
\cite{bsehn14} for the case of $A \in \R^{m \times n}$.
Let $F$ be a facet of a spherical cone $S_v$. Let $y$ be
the vertex of $S_v$ not contained in the $(n-1)$
dimensional facet $F$. Let $Q$ be the
convex hull of $F$ and $y$. We have $Q \subseteq S_v$ 
because $S_v$ is convex. Let $h_F$ be the Euclidean
distance of $y$ from the hyperplane containing $F$,
we have 
\[
vol(S_v) \ge vol(Q) = \frac{area(F) \cdot h_F}{n}.
\]
This yields
\begin{equation}
\frac{D(S_v)}{vol(S_v)}
=\sum_{\mbox{\footnotesize{facet} }F}  \frac{area(F)}{vol(S_v)}
\le n \sum_{\mbox{\footnotesize{facet} }F} \frac{1}{h_F}.
\label{key1}
\end{equation}
Let $a_1, a_2, \ldots, a_n$ be the row vectors of $A_v$,
and $b_1, b_2, \ldots, b_n$ be the column vectors of the
adjugate of $A_v$. Clearly $A_v [b_1, b_2, \ldots, b_n]
=\det(A_v) I$, where $I$ is an identity matrix. This means
that $a_1^{\Tr}b_1=\det(A_v)$ and $b_1 \perp 
\{ a_2, \ldots, a_n \}$. Without loss of generality,
assuming that $y$ lies on the ray generated by $a_1$
(actually $y=a_1$ because we assumed that the lengths
of rows of $A$ are one), clearly, $h_F$ is length of the projection
of $a_1$ onto $b_1$. Let $\theta$ be the angle described
in Figure 1, noticing that the absolute value of 
each component of $b_1$ is less than or equal to $\Delta$,
we have
\[
h_F = \| a_1 \| \cos(\theta) = \| a_1 \|
\frac{a_1^{\Tr}b_1}{\| a_1 \|\cdot \| b_1 \|}
=\frac{a_1^{\Tr}b_1}{\| b_1 \|}
=\frac{\det(A_v)}{\| b_1 \|}
\ge \frac{\det(A^*)}{ \sqrt{n} \Delta }.
\]
Substituting this into (\ref{key1}) completes the proof.
\hfill \qed
\end{proof}

The aforementioned two lemmas lead to the following
claim.

\begin{lemma}
Let $P=\{ x\in \R^{n}: Ax \le b \}$ be a general 
polytope with $A \in \R^{ m \times n}$ and $m>n \ge 2$. 
Assume that all $(n-1) \times (n-1)$ sub-determinants of
$A$ are bounded above by $\Delta$ and $\det(A_v)$ are 
bounded below by $\det(A^*)$. Let $I_j \subseteq V$ be
a set of vertices with $vol (I_j) \le \frac{1}{2} vol(B_n)$.
Then the volume of the neighbourhood of $I_j$, denoted
by $vol(S_{N({I_j})})$, satisfies
\begin{equation}
vol(S_{N({I_j})}) \ge  
\sqrt{\frac{2}{\pi}} \frac{\det(A^*)}{n^2 \Delta }
 \cdot  vol(S_{I_j}).
\label{key2}
\end{equation}
\label{lemma3}
\end{lemma}
\begin{proof}
Noticing that $D(S_{I_j})$ is part of 
$\sum_{v \in N(I_j)} D(S_v)$ and using Lemma \ref{lemma1},
we have
\begin{equation}
\sum_{v \in N(I_j)} D(S_v) \ge D(S_{I_j})
\ge  \sqrt{\frac{2n}{\pi}} \cdot  {vol(S_{I_j})}.
\end{equation}
Applying Lemma \ref{lemma2}, we have
\begin{equation}
\sum_{v \in N(I_j)} D(S_v) \le
\frac{n^{2.5} \Delta }{\det(A^*)} 
 \sum_{v \in N(I_j)} vol(S_v) 
=\frac{n^{2.5} \Delta }{\det(A^*)} \cdot vol(S_{N(I_j)}).
\end{equation}
Combining these two inequality gives
\begin{equation}
vol(S_{N(I_j)}) \ge \frac{\det(A^*)}{n^{2.5} \Delta }
\sqrt{\frac{2n}{\pi}}  \cdot {vol(S_{I_j})}
= \sqrt{\frac{2}{\pi}} \frac{\det(A^*)}{n^2 \Delta }
 \cdot  {vol(S_{I_j})}
\end{equation}
This completes the proof.
\hfill \qed
\end{proof}

The main result of this short note follows from Lemma 
\ref{lemma3}.

\begin{theorem}
Let $P=\{ x\in \R^{n}: Ax \le b \}$ be a general 
polytope with $A \in \R^{ m \times n}$ and $m>n \ge 2$. 
Assume that all $(n-1) \times (n-1)$ sub-determinants of
$A$ are bounded above by $\Delta$ and $\det(A_v)$ are 
bounded below by $\det(A^*)$. Then, the diameter of 
the polytope $P$ is bounded by
$\mathcal{O} \left( {n^3 \ln (n) \Delta} \right)$ if 
$\det(A^*) \ge \frac{1}{2}$; for $\det(A^*) < \frac{1}{2}$, 
the bound is given by 
$j= \mathcal{O} \left( \frac{n^3 \Delta }{\det(A^*)} 
\ln  \left( \frac{n}{\det(A^*)} \right) \right)$.
\label{main}
\end{theorem}
\begin{proof}
We assume that the breadth-first-method starts from vertex
$v$. For $j \ge 1$ and $vol(S_{I_{j-1}}) \le \frac{1}{2} 
\cdot vol(B_n)$, using Lemma \ref{lemma3}, we have
\begin{eqnarray}
vol(S_{I_j}) & \ge & \left( 1+\sqrt{\frac{2}{\pi}} 
\frac{\det(A^*)}{n^2 \Delta } \right) \cdot  vol(S_{I_{j-1}})
\nonumber \\
 & \ge & \left( 1+\sqrt{\frac{2}{\pi}} 
\frac{\det(A^*)}{n^2 \Delta } \right)^{j} \cdot vol(S_{I_0}),
\end{eqnarray}
where $S_{I_0} = S_v$ includes a simplex $J_n$ 
spanned by $n+1$ vertices composed of $0$ and $n$ row 
vectors of $A_v$ (see Figure \ref{Lemma2}).
Since the volume of $J_n$ is given by \cite{stein66}
\[
vol(J_n)=\frac{\det(A_v)}{n!} \ge \frac{\det(A^*)}{n!},
\]
we have
\begin{equation}
vol(S_{I_0}) \ge vol(J_n) \ge \frac{\det(A^*)}{n!}.
\label{volumeS}
\end{equation}
Assuming $n$ is even (which is easy to derive the result
but the order of the estimation remains the same for odd 
$n$), we have
\begin{equation}
vol(B_n)= \frac{\pi^{\frac{n}{2}}}{\left(\frac{n}{2}\right)!}.
\end{equation}
The condition $vol(S_{I_j}) \le \frac{1}{2} \cdot vol(B_n)$
implies
\[
\frac{1}{2} \cdot vol(B_n) 
= \frac{1}{2}  \frac{\pi^{\frac{n}{2}}}{\left(\frac{n}{2}\right)!}
 \ge vol(S_{I_j}) 
%\nonumber \\
%& \ge & \left( 1+\sqrt{\frac{2}{\pi}} 
%\frac{\det(A^*)}{n^2 \Delta } \right) \cdot  vol(I_{j-1})
%\nonumber \\
% & \ge & 
 \ge \left( 1+\sqrt{\frac{2}{\pi}} 
\frac{\det(A^*)}{n^2 \Delta } \right)^{j} 
\frac{\det(A^*)}{n!},
\]
or
\begin{eqnarray}
\pi^{\frac{n}{2}} \ge 
2\frac{\left(\frac{n}{2}\right)!}{n!}\det(A^*) 
\left( 1+\sqrt{\frac{2}{\pi}} 
\frac{\det(A^*)}{n^2 \Delta } \right)^{j} 
\label{bbasic}
\end{eqnarray}
For $0 \le c \le 1$, it has $\ln (1+c) \ge c/2$. Therefore,
we can rewrite (\ref{bbasic}) as
\begin{eqnarray}
\frac{n}{2} \ln \pi &  \ge & 
\ln \left(2\frac{\left(\frac{n}{2}\right)!}{n!}\det(A^*) \right)
 + j \ln \left( 1+\sqrt{\frac{2}{\pi}} 
\frac{\det(A^*)}{n^2 \Delta } \right)
\nonumber \\
&  \ge &   \ln \left( \frac{1}{n^{n/2}}  \right)
+\ln (2\det(A^*) ) 
+j  \sqrt{\frac{1}{2\pi}} 
\frac{\det(A^*)}{n^2 \Delta }.
\end{eqnarray}
Therefore, we have
\[
\frac{n}{2} \ln (n\pi) \ge \ln (2\det(A^*) ) 
+j  \sqrt{\frac{1}{2\pi}} \frac{\det(A^*)}{n^2 \Delta }.
\]
This shows 
$j= \mathcal{O} \left( \frac{n^3 \Delta \ln (n\pi)}{\det(A^*)} \right)$
if $\det(A^*) \ge \frac{1}{2}$. For $\det(A^*) < \frac{1}{2}$,
\begin{eqnarray}
& & \frac{n}{2} \ln (n\pi) \ge \frac{n}{2}  \ln (2\det(A^*) ) 
+j  \sqrt{\frac{1}{2\pi}} \frac{\det(A^*)}{n^2 \Delta }
\nonumber \\
& \Longrightarrow &
\frac{n}{2} \ln  \left( \frac{n\pi}{2\det(A^*)}  \right)
\ge j  \sqrt{\frac{1}{2\pi}} \frac{\det(A^*)}{n^2 \Delta },
\end{eqnarray}
this shows 
$j= \mathcal{O} \left( \frac{n^3 \Delta }{\det(A^*)} 
\ln  \left( \frac{n}{\det(A^*)} \right) \right)$.
\hfill \qed
\end{proof}

\begin{remark}
The upper bound in Theorem \ref{main} is not only related 
to $n$, like the ones of 
%Kalai-Kleitman, Sukegawa, and Todd 
\cite{kk92,sukegawa17,todd14}, but also
to the condition numbers of the vertices of $A_v$.
If the rays of all $S_v$ are almost perpendicular, then
$\det(A^*)$ will be close to one. Otherwise, if for some
$v$, the rays of $S_v$ are almost linear dependent, then
$\det(A^*)$ will be close to zero, and the diameter 
bound of the polytope given in Theorem \ref{main}
will increase significantly. Therefore,
$\det(A^*)$ can be viewed as the condition number of
the polytope. 
\end{remark}

%If $A \in \Z^{m\times n}$ (noticing that if the rows of $A$ 
%are not normalized, the proof and the result of Theorem 
%\ref{main} are still valid), since the absolute value of
%any non-zero sub-determinant of $A$ is a positive
%integer, we have
%\begin{corollary}
%Let $P=\{ x\in \R^{n}: Ax \le b \}$ be a  
%polytope with $A \in \Z^{ m \times n}$ and $m>n \ge 2$. 
%Assume that all $(n-1) \times (n-1)$ sub-determinants of
%$A$ are bounded above by $\Delta$. Then, the diameter of 
%the polytope $P$ is bounded by
%$\mathcal{O} \left( {n^3 \ln (n) \Delta} \right)$.
%\label{corollary}
%\end{corollary}
%
%\begin{remark}
%The bound $\mathcal{O} \left( {n^3 \ln (n) \Delta} \right)$ 
%of the corollary is an improvement of the bound 
%$\mathcal{O} \left( {n^{3.5} \Delta^2 \ln (n\Delta) } \right)$
%of \cite{bsehn14}.
%\end{remark}

We conclude this short note by examining a high dimensional
cubic polytope.
\begin{example}
\begin{eqnarray}
\left[ 
\begin{array}{cccccc}
1 & 0 & 0 & \ldots & 0 & 0 \\
0 & 1 & 0 & \ldots & 0 & 0 \\
0 & 0  & 1 & \ldots & 0 & 0 \\
\vdots &  \vdots & \vdots &  \ddots & 0 & 0 \\
0 &  0 & 0 & \ldots  & 1 & 0 \\
0 &  0 & 0 & \ldots &  0 & 1 \\
-1 & 0 & 0 & \ldots & 0 & 0 \\
0 & -1 & 0 & \ldots & 0 & 0 \\
0 & 0  & -1 & \ldots & 0 & 0 \\
\vdots &  \vdots & \vdots &  \ddots & 0 & 0 \\
0 &  0 & 0 & \ldots  & -1 & 0 \\
0 &  0 & 0 & \ldots &  0 & -1 
\end{array}
\right]
\left[ \begin{array}{c}
x_1 \\ x_2 \\ \vdots \\  \vdots \\ x_{m-1} \\ x_m
\end{array} \right] 
\le 
\left[ \begin{array}{c}
1 \\ 1 \\ \vdots \\  \vdots \\ 1 \\ 1 \\
0 \\ 0 \\ \vdots \\  \vdots \\ 0 \\ 0
\end{array} \right] .
\label{1stProblem}
\end{eqnarray} 
Clearly, all rows are normalized and the absolute value 
of any sub-determinant of $A$ is either $0$ or $1$.
Applying Theorem 10 of \cite{bsehn14} gives an
upper bound of $\mathcal{O} \left( n^{3.5}  \ln(n) \right)$;
applying Theorem \ref{main} gives an
upper bound of $\mathcal{O} \left( n^{3}  \ln(n) \right)$.
Therefore, for this problem, the new bound proposed in this note
is better than the one in Theorem 10 of \cite{bsehn14}.
However, this example does not imply that the derived 
bound is better than Bonifas et al’s 	in general.
\end{example}

\section{acknowledgment}
This author thanks the anonymous reviewers for 
their valuable comments.

\section{Declarations of interest:} 
This research did not receive any specific grant from 
funding agencies in the public, commercial, or
not-for-profit sectors.

\end{document}